\documentclass[11pt]{amsart}
\usepackage{amsmath,amssymb,amsthm}
\usepackage[colorlinks, bookmarks=true]{hyperref}
\usepackage{color,graphicx,shortvrb}

\usepackage{enumerate}
\usepackage{amsmath}
\usepackage{amssymb}

\usepackage[latin 1]{inputenc}

\title{Quadratic Conorm and extremally rich JB$^*$-triples}

\author[F.B. Jamjoom]{Fatmah B. Jamjoom}
\address{Department of Mathematics, College of Science, King Saud University, P.O.Box 2455-5, Riyadh-11451, Kingdom of Saudi Arabia.}
\email{fjamjoom@ksu.edu.sa}

\author[A.M. Peralta]{Antonio M. Peralta}
\address{Departamento de An{\'a}lisis Matem{\'a}tico, Universidad de Granada,\\
Facultad de Ciencias 18071, Granada, Spain}
\curraddr{Visiting Professor at Department of Mathematics, College of Science, King Saud University, P.O.Box 2455-5, Riyadh-11451, Kingdom of Saudi Arabia.}
\email{aperalta@ugr.es}

\author[A.A. Siddiqui]{Akhlaq A. Siddiqui}
\address{Department of Mathematics, College of Science, King Saud University, P.O.Box 2455-5, Riyadh-11451, Kingdom of Saudi Arabia.}
\email{asiddiqui@ksu.edu.sa}

\author[H.M. Tahlawi]{Haifa M. Tahlawi}
\address{Department of Mathematics, College of Science, King Saud University, P.O.Box 2455-5, Riyadh-11451, Kingdom of Saudi Arabia.}
\email{htahlawi@ksu.edu.sa}

\keywords{Extremally rich JB$^*$-triple, Brown-Pedersen quasi-invertibility, reduced minimum modulus, conorm, quadratic conorm}

\subjclass{Primary 46L70; 17C65; 46L05  Secondary 15A09 47A05 47D25 }

\begin{document}

\thanks{The authors extend their appreciation to the Deanship of Scientific Research at King Saud University for funding this work through research group no RG-1435-020. Second author also partially supported by the Spanish Ministry of Economy and Competitiveness, D.G.I. project no. MTM2011-23843.}

\begin{abstract} We introduce and study the class of extremally rich JB$^*$-triples. We establish new results to determine the distance from an element $a$ in an extremally rich JB$^*$-triple $E$ to the set $\partial_{e} (E_1)$ of all extreme points of the closed unit ball of $E$. More concretely, we prove that $$\hbox{dist} (a,\partial_e (E_1)) =\max \left\{ 1 , \|a\|-1\right\},$$ for every $a\in E$ which is not Brown-Pedersen quasi-invertible. As a consequence,
we determine the form of the $\lambda$-function of Aron and Lohman on the open unit ball of an extremally rich JB$^*$-triple $E$, by showing that $\lambda (a)= \frac12$ for every non-BP quasi-invertible element $a$ in the open unit ball of $E$. We also prove that for an extremally rich JB$^*$-triple $E$, the quadratic connorm $\gamma^{q}(.)$ is continuous at a point $a\in E$ if, and only if, either $a$ is not von Neumann regular {\rm(}i.e. $\gamma^{q}(a)=0${\rm)} or $a$ is Brown-Pedersen quasi-invertible. \end{abstract}

\maketitle

\newtheorem{thm}{Theorem}[section]
 \newtheorem{corollary}[thm]{Corollary}
 \newtheorem{lemma}[thm]{Lemma}
 \newtheorem{prop}[thm]{Proposition}
 \theoremstyle{definition}
 \newtheorem{defn}[thm]{Definition}
 \theoremstyle{remark}
 \newtheorem{remark}[thm]{Remark}
 \newtheorem*{Example}{Example}
 \numberwithin{equation}{section}
\newcommand{\vx}{{\mathcal V}(x)}
\newcommand{\sx}{{\mathcal S}(x)}

\section{Introduction}

This paper presents new investigations which provide some answers to problems concerning the geometric structure of those complex Banach spaces included in the class of JB$^*$-triples. W. Kaup proved in 1983 that the open unit ball of a complex Banach space $X$ is a bounded symmetric domain (a pure holomorphic property) if, and only, if $X$ is a JB$^*$-triple (cf. \cite{Ka83}). That is, the holomorphic properties of the open unit ball of $X$ determine when $X$ satisfies certain algebraic-geometric axioms listed below. L. Harris proved in \cite{Harris74} that the open unit ball of a C$^*$-algebra $A$ is a bounded symmetric domain; actually $A$ is a JB$^*$-triple with triple product defined by \begin{equation}\label{eq triple product C*} \{x,y,z\}: =\frac12 (xy^*z+zy^* x).
\end{equation}

JB$^*$-triples have been intensively studied during the last three decades, and special attention has been paid to the geometric properties of these spaces. In several cases the studies determine whether JB$^*$-triples satisfy certain properties fulfilled by C$^*$-algebras. For example, the quasi-invertible elements of C$^*$-algebras studied by L.G. Brown and G.K. Pedersen in \cite{BP95} gave rise to the introduction of Brown-Pedersen quasi-invertible elements in a  JB$^*$-triple in \cite{SiTah2011} and \cite{TahSiddJam2013}. Brown and Pedersen show in \cite{BP97} that quasi-invertible elements in C$^*$-algebras play a crucial role to determine the form  of the $\lambda$-function of R. Aron and R. Lohman \cite{AronLohman87}. The precise description of the $\lambda$-function is determined in \cite{BP97}. A C$^*$-algebra $A$ is said to be extremally rich if the set $A^{-1}_q$ of quasi-invertible elements in $A$ is norm-dense in $A$ (cf. \cite[\S 3]{BP95}). The class of extremally rich C$^*$-algebras is strictly bigger than the class of von Neumann algebras. From the geometric point of view, a unital C$^*$-algebra $A$ is extremally rich if and only if it has the (uniform) $\lambda$-property, that is, the infimum of the values of the $\lambda$-function on the closed unit ball of $A$ is bigger than zero (cf. \cite[\S 3]{BP95} and \cite[Theorem 3.7]{BP97}).\smallskip

In a recent paper, we proved that every JBW$^*$-triple (i.e. a JB$^*$-triple which is also a dual Banach space) satisfies the uniform $\lambda$-property (cf. \cite{JamPerSiTah2015}). In the same paper we determine the $\lambda$-function of the closed unit ball of a JBW$^*$-triple, and on the set $E_q^{-1}$ of all Brown-Pedersen quasi-invertible elements in the closed unit ball of a general JB$^*$-triple $E$. If we also assume that the set $\partial_{e} (E_1)$ of all extreme points in the closed unit ball $E_1$ of $E$ is non-empty, we can only prove that the inequality \begin{equation}\label{eq 0 inequality lambda 3.6} \lambda (a)\leq \frac12 (1-\alpha_q (a)),
 \end{equation} holds for every $a\in E_1\backslash E_q^{-1}$, where $\alpha_q (a)$ is the distance from $a$ to $E_q^{-1}$ (cf. \cite[Corollary 3.7]{JamPerSiTah2015}).\smallskip

The question whether in \eqref{eq 0 inequality lambda 3.6} the inequality sign can be replaced with an equality symbol, is one of the main open problems in the setting of JB$^*$-triples. This question is deeply related to the problem of determining the distance from an element $a$ to the set $\partial_{e} (E_1)$. The best estimation follows from Theorem 3.6 in \cite{JamPerSiTah2015}, where it is established that for every JB$^*$-triple  $E$ with $\partial_{e} (E_1)\neq  \emptyset,$ the inequalities $$1+\|a\|\geq \hbox{dist} (a, \partial_{e} (E_1)) \geq \max \left\{ 1+ \alpha_q (a) ,  \|a\|-1  \right\},$$ hold for every $a$ in $E\backslash E_q^{-1}.$
\smallskip

We introduce, in this paper, the notion of extremally rich JB$^*$-triple with the aim of studying the above problems more in deep. We say that a JB$^*$-triple $E$ is extremally rich if the set of Brown-Pedersen quasi-invertible elements in $E$ is norm dense in $E$. Several characterizations of extremally rich JB$^*$-triples are provided in Proposition \ref{p charact of extremally rich}. Among the new results in this paper, we prove that for an extremally rich JB$^*$-triple we have $$\hbox{dist} (a,\partial_e (E_1)) =\max \left\{ 1 , \|a\|-1\right\},$$ for every $a\in E\backslash E_q^{-1}$ (see Theorem \ref{thm extremally rich JB*-triples satisfy lambda property on the open unit ball}). As a consequence, we show that $\lambda (a)= \frac12$ for every non-BP quasi-invertible element $a$ in the open unit ball of an extremally rich JB$^*$-triple (see Corollary \ref{c lambda function on the open unit ball of a extremally rich}).\smallskip

We also deal with another related, open question. We recall that the \emph{reduced minimum modulus} of a non-zero bounded linear or conjugate linear operator $T$ between two normed spaces $X$ and $Y$ is defined by
\begin{equation}\label{red min mod} \gamma (T) := \inf \{ \|T(x)\| : \hbox{dist}(x,\ker(T)) \geq 1 \}.\end{equation} Following  \cite{Ka}, we set
$\gamma (0) = \infty$. When $X$ and $Y$ are Banach spaces, we have $\gamma (T) >0
\Leftrightarrow T(X)$  is norm closed,  (cf. \cite[Theorem IV.5.2]{Ka}).\smallskip

The quadratic-conorm, $\gamma^{q} (a),$ of an element $a$ in a JB$^*$-triple $E$ is defined as the reduced minimum modulus of the conjugate linear operator $Q(a): E\to E,$ $x\mapsto Q(a) (x) := \{a,x,a\}$, that is, $\gamma^{q} (a) = \gamma (Q(a))$ (see \cite{BurKaMoPeRa}). Theorem 3.13 in \cite{BurKaMoPeRa} proves that $\gamma^{q}(.)$ is upper semi-continuous on $E\backslash \{0\}$. It is also remarked, in the just quoted reference, that the continuity points of $\gamma^{q} (.)$ are, in general, unknown. In this paper we throw some light into the question of determining the continuity points of the quadratic conorm, by showing that for an extremally rich JB$^*$-triple $E$, the quadratic connorm $\gamma^{q}(.)$ is continuous at a point $a\in E$ if, and only if, either $a$ is not von Neumann regular {\rm(}i.e. $\gamma^{q}(a)=0${\rm)} or $a$ is BP quasi-invertible (see Theorem \ref{t characterization of continuity points of conorm}). We also explore the applications of this result to determine the continuity points of the conorm of an extremally rich C$^*$-algebra in the sense introduced by R. Harte and M. Mbekhta in \cite{HarMb2}.

\subsection{Preliminaries}

A complex Banach space $E$ is a JB$^*$-triple if it can be equipped with a triple product $\{.,.,.\} : E\times E\times E \to E$, $(x,y,z)\mapsto \{x,y,z\},$ which is linear and symmetric in $x$ and $z$ and conjugate linear in $y$, and satisfies the following axioms:\begin{enumerate}[$(a)$] \item (Jordan identity)
$$\{x,y,\{a,b,c\}\} = \{\{x,y,a\},b,c\} - \{a,\{y,x,b\},c\} + \{a,b,\{xyc\}\},$$ for every $a,b,c\in E$;
\item For each $a\in E$, the operator $x\mapsto L(a,a) (x):= \{a,a,x\}$ is hermitian with non-negative spectrum;
\item $\|\{x,x,x\}\| = \|x\|^{3}$, for all $x\in E$.
\end{enumerate}

The class of JB$^*$-triples includes all C$^*$-algebras, all complex Hilbert spaces, and all spin factors. It is further known that every JB$^{*}$-algebra is a JB$^*$-triple with the triple product $\{x,y,z\} := (x{\circ}y^{*}){\circ}z-(x{\circ}z){\circ}y^{*}+(y^{*}{\circ}z){\circ}x$.\smallskip

A \emph{JBW$^*$-triple} is a JB$^*$-triple which is also a dual Banach space. The the second dual $E^{**}$ of a JB$^*$-triple $E$ is JBW$^*$-triple (cf. \cite{Di86b}). Every JBW$^*$-triple $W$ admits a unique isometric predual $W_*$, and its triple product is separately $\sigma(W,W_*)$-continuous (cf. \cite{BarTi}).\smallskip

JB$^*$-triples are stable by $\ell_{\infty}$-sums (cf. \cite[page 523]{Ka83}), that is, if $(E_j)$ is a family of JB$^*$-triples, then the $\ell_{\infty}$-sum $\displaystyle \oplus_{j}^{\infty} E_j$ is a JB$^*$-triple with respect to the product $$\{(a_j),(b_j),(c_j)\}:= (\{a_j,b_j,c_j\}).$$

Following standard notation, given two elements $x,y$ in a JB$^*$-triple $E$, the conjugate linear operator $Q(x,y): E\to E$ is defined by $Q(x,y)z:=\{x,z,y\}$ for all $z\in E$; we usually write $Q(x)$ instead of $Q(x,x)$. The symbol $L(x,y)$ will stand for the linear operator on $E$ given by $L(x,y) (z) = \{x,y,z\}$.\smallskip

We recall that an element $e$ in a JB$^*$-triple $E$ is said to be a \emph{tripotent} if $\{e,e,e\} =e$. It is known that, for each tripotent $e$ in $E$, the eigenvalues of the operator $L(e,e)$ are contained in the set $\{0, \frac12, 1\}$, and $E$ decomposes in the form $$ E=
 E_{2} (e) \oplus  E_{1} (e) \oplus  E_0 (e),$$ where, for $i=0,1,2,$  $E_i (e)$ is the $\frac{i}{2}$-eigenspace of $L(e,e)$. This decomposition is called the \emph{Peirce decomposition} of $E$ with respect to $e$. The Peirce subspaces appearing in the above decomposition satisfy certain multiplication rules (called Peirce rules), which can be stated as follows: $$\{ E_{i}(e), E_{j} (e), E_{k} (e)\}\subseteq E_{i-j+k} (e)$$ if $i-j+k \in \{ 0,1,2\}$ and is zero otherwise. In addition, $$ \{E_{2}(e), E_{0}(e), E\}=0.$$ The projection of $E$ onto $E_{k} (e)$ is denoted by $P_{k_{}}(e)$, and it is called the Peirce $k$-projection. Peirce projections are contractive (cf. \cite{FriRu85}) and satisfy that $P_{2}(e) = Q(e)^2,$ $P_{1}(e) =2(L(e,e)-Q(e)^2),$ and $P_{0}(e) =Id_E - 2 L(e,e) + Q(e)^2.$ A tripotent $e$ in $E$ is said to be \emph{unitary} if $L(e,e)$ coincides with the identity map on $E,$ that is, $E_2 (e) = E$. If $E_0 (e) =\{0\}$ we say that $e$ is \emph{complete}.\smallskip

The Peirce space $E_2 (e)$ is a unital JB$^*$-algebra with unit $e$, product $x\circ_e y := \{x,e,y\}$ and involution $x^{*_e} := \{ e,x,e\}$, respectively.
Furthermore, the triple product on $E_2 (e)$ is given by $$\{ a,b,c\} = (a \circ_e b^{*_e}) \circ_e c + (c\circ_e b^{*_e}) \circ_e a - (a\circ_e c)
\circ_e b^{*_e} \ (a,b,c\in E_{2} (e)).$$

Let $a$ be an element in a JB$^*$-triple $E$, and let $E_a$ denote the JB$^*$-subtriple of $E$ generated by $a$. That is, $E_a$ coincides with the closed linear span of the elements $a$, $a^{[3]}= \{a,a,a\}$, $a^{[2n+1]}:= \{a,a,a^{[2n-1]}\}$ ($n\geq 2$). It follows from the \emph{commutative Gelfand theory} that there exist a locally compact Hausdorff space $L_{a}\subseteq (0,\|a\|],$ with $L_a \cup \{0\}$ compact, and a triple isomorphism $\Psi_a : E_a \to C_0(L_a)$, where $C_0 (\Omega_{x})$ denotes the Banach space of all complex-valued continuous functions vanishing at $0,$ such that $\Psi_a (a) (t) = t,$ $\forall t\in L_a$ (cf. \cite[\S 1, Corollary 1.15]{Ka83}). Therefore, for each natural $n$, there exists a unique $a^{[{1}/({2n-1})]}\in E_a$ satisfying $(a^{[{1}/({2n-1})]})^{[2n-1]} = a$. The sequence $(a^{[{1}/({2n-1})]})$ need not be convergent in the norm topology of $E$. However, when $a$ is an element in a JBW$^*$-triple $W$, the
sequence $(a^{[{1}/({2n-1})]})$ converges in the weak$^*$
topology of $W$ to a tripotent in $W$, which is denoted by $r(a)$ and
is called the \emph{range tripotent} of $a$. The tripotent $r(a)$ is
the smallest tripotent $e$ in $W$ satisfying that $a$ is positive in
the JBW$^*$-algebra $W_{2} (e)$ (cf. \cite[\S 3, Lemma~3.2]{EdFerHosPe2010}). \smallskip

The deep geometric-algebraic connections appearing in the setting of JB$^*$-triples materialize in many important properties, one of them ensures that complete tripotents in a JB$^*$-triple $E$ coincide with the extreme points of its closed unit ball (cf. \cite[Lemma 4.1]{BraKaUp78} and \cite[Proposition 3.5]{KaUp77} or \cite[Theorem 3.2.3]{Chu2012}). Throughout this paper, the set of all the extreme points of the closed unit ball of a Banach space $X$ is denoted by $\partial (X_1)$.\smallskip

The Bergman operator, $B(x,y)$, associated with a couple of elements $x,y\in E$ is the mapping defined by $B(x,y):=I_{E}- 2L(x,y)+ Q(x)Q(y)$, where $I_{E}$ is the identity operator on $E$. We observe that, for a tripotent $e\in E$, $B(e,e) = P_0(e)$. According to \cite{TahSiddJam2013}, an element $x$ in a JB$^*$-triple ${E}$ is called \emph{Brown-Pedersen quasi-invertible} (BP quasi-invertible for short) if there exists $y\in E$ satisfying $B(x,y)=0$. We say that $y$ is a BP quasi-inverse of $x$. It is known that $B(x,y)=0\Rightarrow B(y,x)=0$. A BP quasi-invertible element need not admit a unique BP quasi-inverse. It is established in \cite{TahSiddJam2013} that an element $x$ in $E$ is BP quasi-invertible if and only if there exits a complete tripotent $v\in E$ ($v\in\partial_e (E_1)$)  such that $x$ is positive and invertible in the Peirce $2$-space $E_{2}(v)$. Therefore, the set $E^{-1}_{q}$ of all BP quasi-invertible elements in $E$ contains all extreme points of the closed unit ball of $E$. When ${E}=\mathcal{J}$ is a JB$^*$-algebra, $\mathcal{J}^{-1}_{q}$ contains the set $\mathcal{J}^{-1}$ of all invertible elements in ${E}$.\smallskip

When a C$^*$-algebra $A$ is regarded as a JB$^*$-triple, BP quasi-invertible elements in $A$ are precisely the quasi-invertible elements of $A$ in the sense defined by Brown and Pedersen in \cite{BP95}.\smallskip

We also recall that an element $a$ in a JB$^*$-triple $E$ is called \emph{von Neumann regular} if and only if there exists $b\in E$ such that $Q(a)b =a,$ $Q(b)a =b$ and $[Q(a),Q(b)]:=Q(a)\,Q(b) - Q(b)\, Q(a)=0$ (cf.  \cite[Lemma 4.1]{Ka96}). For a von Neumann regular element $a$, there might exist many elements $c$ in $E$ such that $Q(a)c =a$. However, there exists a unique element $b \in E$ (denoted by $a^{\dag}$) satisfying $Q(a)b =a,$ $Q(b)a =b$ and $[Q(a),Q(b)]=0$, this unique element $b$ is called the \emph{generalized inverse} of $a$ in $E$.  For an element $a$ in a JB$^*$-triple $E$, we can consider the range tripotent, $r(a),$ of $a$ in $E^{**}$. It is known that $a$ is von Neumann regular if, and only if, $r(a)\in E$ and $a$ is positive and invertible in $E_2 (r(a))$ (cf. \cite[\S \S 2, pages 191 and 192]{BurKaMoPeRa}).

\section{Extremally Rich JB$^*$-triples}

In \cite[\S 3]{BP95} L. Brown and G.K. Pedersen introduced and studied the class of extremally rich C$^*$-algebras. We recall that a unital C$^*$-algebra $A$ is \emph{extremally rich} if the set $A_q^{-1}$ of Brown-Pedersen quasi-invertible elements in $A$ is (norm) dense in $A$. When $A$ is non-unital, it is extremally rich if its unitization enjoys this property. Every von Neumann algebra and every purely infinite (simple) C$^*$-algebra is extremally rich (cf. \cite[\S 3]{BP95}). From the point of view of Banach space theory, a unital C$^*$-algebra is extremally rich if and only if it has the (uniform) $\lambda$-property defined by R. Aron and R. Lohman in \cite{AronLohman87} (cf. \cite[\S 3]{BP95} and \cite[Theorem 3.7]{BP97}).\smallskip

We recall that, given a normed space $X$, $x,y \in X$, with $\|y\|\leq 1$, $e\in \partial_e (X_1)$, and $0<\lambda \leq 1,$ the ordered $3$-tuple $(e, y, \lambda)$ is said to be \emph{amenable} to $x$ if $x= \lambda e + (1-\lambda) y$. The $\lambda$-function is defined by $$\lambda (x) := \sup \{\lambda : (e, y, \lambda) \hbox{ is a $3$-tuple amenable to } x\}.$$ The space $X$ is said to have the $\lambda$-property if each element in its closed unit ball admits an amenable $3$-tuple (cf. \cite{AronLohman87}).\smallskip

The notion of Brown-Pedersen quasi-invertibility in JB$^*$-triples have been recently studied in \cite{SiTah2011}, \cite{TahSiddJam2013} and \cite{TahSiddJam2014}. The study of the $\lambda$-function in JB$^*$-triples is developed in \cite{TahSiddJam2014} and \cite{JamPerSiTah2015}, where we proved that every JBW$^*$-triple (i.e. a JB$^*$-triple which is a dual Banach space) satisfies the (uniform) $\lambda$-property. We introduce the following definition with the aim of determining those JB$^*$-triples satisfying the (uniform) $\lambda$-property.

\begin{defn}\label{D:extremrich} A JB$^*$-triple $E$ is called \emph{extremally rich} if the set $E^{-1}_{q}$ of BP quasi-invertible elements in $E$ is norm dense in $E$.
\end{defn}

Recall that an element $u$ in a unital JB$^*$-algebra $\mathcal{J}$ is called \emph{unitary} if $u^{*} = u^{-1}$ (where $u^{-1}$ denotes the inverse of $u$), or equivalently, if $\{u,u,z\} = z$, $\forall z \in \mathcal{J}$ (cf. \cite[Definition 19.12]{Up} or \cite[Definition 4.1.53, Proposition 4.1.54 4.1.55]{CabRod2014}), that is, $L(u,u) = I_{\mathcal{J}}$ (the identity operator over $\mathcal{J}$).

\begin{remark}\label{r tsr1} $(a)$ We recall that a unital C$^*$-algebra $A$ is of \emph{topological stable rank 1}  (tsr $1$) if the subgroup $A^{-1}$ of invertible elements in $A$ is norm dense in $A$ (see \cite{Rieffel83}). A similar definition is introduced in the category of JB$^*$-algebras in \cite{S1}.\smallskip

If $\mathcal{J}$ is a JB$^*$-algebra of tsr $1$, then $\mathcal{J} = \overline{\mathcal{J}^{-1}} \subseteq \overline{\mathcal{J}^{-1}_{q}} $. This shows that every JB$^*$-algebra $\mathcal{J}$ of $tsr\  1$ is extremally rich. There exist examples of extremally rich C$^*$-algebras which are not of tsr $1$. For example, suppose $A$ is a von Neumann algebra that contains a non-unitary, maximal partial isometry, say $v$, which is a non-unitary extreme point of $A_{1}$.
Then, $v\in \partial_e(A_{1}) \neq \mathcal{U}(A),$ which implies that $A$ is not of tsr $1$ (cf. \cite[Corollary 6.10]{S1}). On the other hand, every von Neumann algebra is extremally rich (see \cite[p.126]{BP95}).\smallskip

$(b)$ It should be also noted that the von Neumann envelope of a JB$^*$-algebra of tsr $1$ need not be, in general, of tsr $1$ (cf. \cite[Theorems 3.1 and 3.2]{S1}).\smallskip

$(c)$ Let $A$ be a C$^*$-algebra. Then $A$ is extremally rich as a C$^*$-algebra if and only if $A$ is extremally rich when it is regarded as a JB$^*$-triple with the product in \eqref{eq triple product C*}.
\end{remark}

Since the class of extremally rich C$^*$-algebras is strictly bigger than the class of von Neumann algebras, we can immediately confirm that the class of extremally rich JB$^*$-triples is agreeably large, strictly bigger than the class of JBW$^*$-triples.\smallskip

In our next result we shall establish some characterizations of extremally rich JB$^*$-triples in the line set down by Brown and Pedersen for C$^*$-algebras in \cite[Theorem 3.3]{BP95}. To this end, we shall recall a result taken from \cite[Proposition 4.4]{JamPerSiTah2015}. First, we recall that for each element $a$ in a JB$^*$-triple $E$, $m_q (a):=\hbox{dist} (a,E\backslash E_q^{-1})$ coincides with the square root of the quadratic conorm of $a$, whenever $a$ is in $E_q^{-1}$ (see \cite[Theorem 3.1]{JamPerSiTah2015}).\smallskip

\begin{prop}\label{propo 4.4 in JamPeSidTah}\cite[Proposition 4.4]{JamPerSiTah2015}
Let $a$ and $b$ be elements in a JB$^*$-triple $E$. Suppose $\|a-b\| < \beta$ and $b \in E_{q}^{-1}$. Then $a + \beta r(b)\in E_{q}^{-1}$ and the inequality $$ m_{q} (a+ \beta r(b)) \geq \beta -\|b-a\|,$$ holds. Furthermore, under the above conditions, the element $P_{2} (r(b))(a) +  \beta r(b)$ is invertible in the JB$^*$-algebra $E_{2}(r(b))$.$\hfill\Box$
\end{prop}

The promised characterization of extremally rich JB$^*$-triples reads as follows:

\begin{prop}\label{p charact of extremally rich} For a JB$^*$-triple $E$ with $\partial_{e} (E_1)\neq \emptyset$, the following conditions are equivalent:\begin{enumerate}[$(i)$]\item $E$ is extremally rich;
\item For every $ a \in E$ and any $\beta > 0$, there is an element $b \in E_{q}^{-1}$, with range tripotent $r(b) \in \partial_{e}(E_{1}),$ satisfying that  $a + \beta r(b)\in E_{q}^{-1}$.
\item For every $ a \in E$ and any $\beta > 0$, there is an element  $b \in E_{q}^{-1}$ such that $P_{2} (r(b)) (a) +  \beta r(b) $ is invertible in the JB$^*$-algebra $E_{2}(r(b))$.
\end{enumerate}
\end{prop}

\begin{proof}
The implications $(i) \Rightarrow (ii)$ and $(i) \Rightarrow (iii)$ follow from the above Proposition \ref{propo 4.4 in JamPeSidTah} (see \cite[Proposition 4.4]{JamPerSiTah2015}). The implication $(ii) \Rightarrow (i)$ is clear from the definition of extremal richness and the arbitrariness of $\beta $.\smallskip

$(iii) \Rightarrow (ii)$ Fix $a\in E$ and $\beta>0$. By assumptions, there exists $b \in E_{q}^{-1}$ such that $P_{2} (r(b)) (a) +  \beta r(b) \in E_{2} (r(b))$ is invertible in the JB$^*$-algebra $E_{2}(r(b))$. Since $r(b)$ is an extreme point of $E$ and $ P_{2}(r(b)) (a + \beta r(b)) = P_{2}(r(b)) (a) + \beta r(b)) $ is invertible in the JB$^*$-algebra $E_{2}(r(b))$, it follows from \cite[Lemma 2.2]{JamPerSiTah2015} that $a + \beta r(b) \in E^{-1}_{q}$, and clearly $\left\| a -(a + \beta r(b)) \right\|= \beta$.
\end{proof}

We shall explore next the stability of the property of being extremally rich under $\ell_{\infty}$-sums, ideals, and quotients. We recall that a (closed) subtriple $I$ of a JB$^*$-triple $E$ is said
to be an \emph{ideal} of $E$ if $\{E,E,I\} + \{E,I,E\}\subseteq I$. It is known that $I$ is an ideal if and only if $\{E,E,I\} \subseteq I$ or $\{ E,I,E\} \subseteq I$ (compare \cite{Bun86}). 

\begin{thm}\label{t:quot sum}
Every quotient of an extremally rich JB$^*$-triple is extremally rich. Let $(E_j)$ be a family of JB$^*$-triples. Then
an element $a= (a_j) \in E$ is BP quasi-invertible if, and only if, $a_j$ is BP quasi-invertible in $E_j$ for every $j$. Consequently, the $\ell_{\infty}$-sum $\displaystyle E=\oplus_{j}^{\infty} E_j$ is an extremally rich JB$^*$-triple if and only if each $E_j$ is extremally rich.
\end{thm}

\begin{proof} Let $E$ be an extremally rich JB$^*$-triple, and let $J$ be a closed ideal of $E$. Let $\pi:E\to E/J,$ $\pi(x) = x+J$, denote the canonical projection of $E$ onto $E/J$. Since $\pi$ is a surjective triple homomorphism, it follows that $\pi (E_q^{-1}) \subseteq (E/J)_q^{-1}$ (cf. \cite[Theorem 3]{TahSiddJam2013}). By hypothesis, there exist a sequence $(x_{n})$ in $E^{-1}_{q}$ converging to $x$ in the norm topology of $E$. Since $\pi$ is continuous, $\pi(x_{n}) \rightarrow \pi(x) $ in norm,  which proves that $\pi (x) \in \overline{(E/J)^{-1}_{q}}$, and hence $\overline{(E/J)^{-1}_{q}} = E/J$.\smallskip

Now, let $(E_{j})_{j \in I}$ be an indexed family of JB$^*$-triples. We set $\displaystyle E= \oplus^{\infty}_j E_{j}$. 
Suppose $a= (a_j) \in E$ is BP quasi-invertible. Since, for each $j_0$, the canonical projection $\displaystyle \pi_{j_0} : \oplus_{j \in I}^{\infty} E_{j} \rightarrow E_{j_0}$ is a surjective triple homomorphism, we deduce that $a_{j_0}= \pi_{j_0} (a) \in (E_{j_0})_q^{-1}$. Suppose now, that $a_j$ is BP quasi-invertible in $E_j$ for every $j$. Consider $b_j\in E_j$ satisfying $B(a_j,b_j)=0$ on $E_j$. Then $B((a_j),(b_j)) =0$ on $E$, which shows that $a = (a_j)$ is BP quasi-invertible in $E$. The final statement follows from the above.
\end{proof}

Theorem 3.6 in \cite{JamPerSiTah2015} establishes that for every JB$^*$-triple  $E$ with $\partial_{e} (E_1)\neq  \emptyset,$ the inequalities $$1+\|a\|\geq \hbox{dist} (a, \partial_{e} (E_1)) \geq \max \left\{ 1+ \alpha_q (a) ,  \|a\|-1  \right\},$$ hold for every $a$ in $E\backslash E_q^{-1}.$ Under the additional hypothesis of $E$ being extremally rich we can obtain an optimal computation of the distance from an element in $E$ to the set $\partial_{e} (E_1)$ of extreme points of $E_1$.

\begin{thm}\label{thm extremally rich JB*-triples satisfy lambda property on the open unit ball} Let $E$ be an extremally rich JB$^*$-triple and $x \in E \backslash E^{-1}_{q}$, then dist$(x, \partial_{e}(E_{1})) = \max\{1 , \|x\| - 1\}$. In particular, if $x \in E_{1} $ then dist$(x,\partial_{e}(E_{1}) )= 1$. Consequently, for each $x$ in $E$ we have $$\hbox{dist} (x,\partial_e (E_1)) =\left\{\begin{array}{lc}
                                                                 \max \left\{ 1- m_q (x) , \|x\|-1\right\}, & \hbox{ if } x\in E_q^{-1}; \\
                                                                 \ & \ \\
                                                                 \max \left\{ 1 , \|x\|-1\right\}, & \hbox{ if } x\notin E_q^{-1},
                                                               \end{array}
 \right.$$ where $m_q (x)=\hbox{dist} (x,E\backslash E_q^{-1})$.
\end{thm}

\begin{proof} Since the JB$^*$-triple $E$ is extremally rich,  $\alpha_q (x) = \hbox{dist} (x,E_q^{-1})= 0$ for all $x \in E$. Theorem 3.6 in \cite{JamPerSiTah2015} implies that $$ \hbox{dist}(x, \partial_{e}(E_{1}) \geq \max\{1+\alpha_q (x), \|x\| - 1\} = \max\{1, \|x\| - 1\} =1.$$ Applying \cite[Theorem 27]{TahSiddJam2013} we obtain
dist$(x, \partial_{e}(E_{1}) \leq \max\{1 , \|x\| - 1\}$  for all  $x \in \overline{E_q^{-1}}= E$. Combining the above inequalities we have $$\hbox{dist} (x, \partial_{e}(E_{1}) = \max\{1 , \|x\| - 1\},$$  for all  $x \in E \backslash E^{-1}_{q}$. The second statement of the theorem follows immediately when $\|x\| \leq 1$.\smallskip

The final statement follows form the first estimation and from \cite[Theorem 3.1 and Proposition 3.2]{JamPerSiTah2015}.
\end{proof}

We have already commented that from the geometric point of view of Banach space theory, a C$^*$-algebra is extremally rich if and only if it has the (uniform) $\lambda$-property (cf. \cite[\S 3]{BP95} and \cite[Theorem 3.7]{BP97}). We do not know if this statement remains true for JB$^*$-triples. We know that every JBW$^*$-triple satisfies the uniform $\lambda$-property (cf. \cite{JamPerSiTah2015}). For an element $a$ in a JB$^*$-triple $E$ we further known that $\lambda (a) = \frac{1+m_q (a)}{2},$ whenever $a$ is BP quasi-invertible element in $E_1$ (\cite[Theorem 3.4]{JamPerSiTah2015}). If we also assume that $\partial_{e} (E_1)\neq  \emptyset,$ the inequalities $$1+\|a\|\geq \hbox{dist} (a, \partial_{e} (E_1)) \geq \max \left\{ 1+ \alpha_q (a) ,  \|a\|-1  \right\},$$ hold for every $a$ in $E\backslash E_q^{-1}$ (\cite[Theorem 3.6]{JamPerSiTah2015}), and hence \begin{equation}\label{eq inequality lambda 3.6} \lambda (a)\leq \frac12 (1-\alpha_q (a)),
 \end{equation}for every $a\in E_1\backslash E_q^{-1}$. We shall prove now that the $\lambda$-function takes only values bigger or equal than $\frac12$ on the open unit ball of an extremally rich JB$^*$-triple.

\begin{corollary}\label{c lambda function on the open unit ball of a extremally rich} Let $a$ be an element in the open unit ball of an extremally rich JB$^*$-triple. Suppose $a$ is not BP quasi-invertible. Then $\lambda (a)= \frac12.$
\end{corollary}

\begin{proof} Let us pick a real number $t<\frac12$. Clearly $\beta = \frac1t >2$. Since $\|a\|<1$ and $a\in E\backslash E_q^{-1}$, we deduce, via Theorem \ref{thm extremally rich JB*-triples satisfy lambda property on the open unit ball}, that $$ \hbox{dist} (\beta a, \partial_{e}(E_{1})) = \max\{ 1 , \beta \|a\| - 1\}< \beta -1. $$ Therefore, there exists $e\in \partial_e (E_1)$ satisfying $\|\beta a -e\| < \beta-1$. The element $y= \frac{1}{\beta-1} (\beta a -e)$ lies in the open unit ball of $E$, and we can write $\beta a = e + (\beta-1) y,$ and hence $a = \frac1{\beta} e + \frac{\beta-1}{\beta} y = t e + (1-t) y$, which proves that $\lambda (a) \geq t$. The arbitrariness of $t$ shows that $\frac12 \leq \lambda(a)$. The final statement follows from \cite[Theorem 3.6]{JamPerSiTah2015}.
\end{proof}

\begin{remark}\label{r JB*-triples with the uniform lambda 12 propety are extremally rich} Let $E$ be a JB$^*$-triple satisfying the uniform $\lambda$-property of Aron and Lohman with $\frac12 \leq \inf\{\lambda (x) : x\in E_1\}$. We can assure, via \eqref{eq inequality lambda 3.6}, that $\alpha_q (a) =0$, for every $a\in E_1 \backslash E_q^{-1}$. This shows that $E$ is extremally rich.
\end{remark}

In \cite[\S 4]{TahSiddJam2014} the authors introduce the so-called $\Lambda$-condition in JB$^*$-triples. A JB$^*$-triple $E$ satisfies the \emph{$\Lambda$-condition} if for every $v\in\partial_{e}(E_{1})$, and every $y\in (E_{2}(v))_{1}\backslash E^{-1}_{q}$ with $\lambda (y) = 0$ we have $\alpha_{q}(y) = 1.$ We shall consider the following stronger variant: A JB$^*$-triple $E$ satisfies the \emph{strong-$\Lambda$-condition} if for each $y\in E_{1}\backslash E^{-1}_{q}$ with $\lambda (y) = 0$ we have $\alpha_{q}(y) = 1.$ Every C$^*$-algebra $A$ satisfies that $\lambda (a) = \frac12 (1-\alpha_q (a))$ for every  $a\in A_1\backslash A_q^{-1}$ (cf. \cite[Theorem 3.7]{BP97}). Therefore every C$^*$-algebra fulfills the strong-$\Lambda$-condition. A similar identity and statement is also valid for every JBW$^*$-triple (\cite[Theorem 4.2]{JamPerSiTah2015}).\smallskip

Clearly, if $E$ satisfies the strong-$\Lambda$-condition, then $\lambda (a) >0$ for every $a\in E_1\backslash E_q^{-1}$ with $\alpha_q (a)<1$. The following result is a consequence of this fact, Corollary \ref{c lambda function on the open unit ball of a extremally rich}, and the comments preceding it (cf. \cite[Theorems 3.4 and 3.6]{JamPerSiTah2015}).

\begin{corollary} Every extremally rich JB$^*$-triple satisfying the strong-$\Lambda$-condition satisfies the $\Lambda$-property of Aron and Lohman.$\hfill\Box$
\end{corollary}

\section{Quadratic conorm in extremally rich JB$^*$-triples}

As mentioned in the introduction, the quadratic conorm, $\gamma^{q} (a),$ of an element $a$ in a JB$^*$-triple $E$ is defined as the reduced minimum modulus of the conjugate linear operator $Q(a)$ (see \cite[Definition 3.1]{BurKaMoPeRa}), that is, $$\gamma^q (a) := \gamma (Q(a)) = \inf \{ \|Q(a) (x)\| : \hbox{dist}(x,\ker(Q(a))) \geq 1 \}.$$ The authors in \cite{BurKaMoPeRa} established that $\gamma^{q}(a) = \frac{1}{\|a^{\dag}\|^{2}},$  whenever $a$ is von Neumann regular (where $a^{\dag}$ is the unique generalized inverse of $a$), and $\gamma^q (a) =0$ otherwise (see  \cite[Theorem 3.4 and its proof]{BurKaMoPeRa}).\smallskip

Theorem 8 in \cite{TahSiddJam2013} asserts that the set $E_q^{-1}$ of all BP quasi-invertible elements in a JB$^*$-triple $E$ is open in the norm topology. A more explicit measure of this fact is given in the next result.

\begin{lemma}\label{l 31}
Let $a$ be a BP quasi-invertible element in a JB$^*$-triple $E$. Suppose
$b$ is an element in $E$  satisfying $\|a - b\| < \gamma^{q}(a)^{\frac{1}{2}}$. Then $b \in E^{-1}_{q}$.
\end{lemma}

\begin{proof} We recall that $a$ being BP quasi-invertible implies that $e=r(a)\in \partial_e (E_1)$, and $a$ is positive and invertible in the JB$^*$-algebra $(E_2 (e),\circ_{e},{*_{e}})$. We further know that its generalized inverse $a^{\dag} \in E_{2}(e)$ coincides with its inverse in this JB$^*$-algebra.\smallskip

Let $c=a^{\frac12}$ denote the square root of $a$ in $E_2(e)$. We observe that $a^{\frac12}$ is positive and invertible in $E_2(e).$ Moreover, the inverse of $a^{\frac12}$, $(a^{\frac12})^{-1},$ coincides with $(a^{\frac12})^{\dag}= (a^\dag)^{\frac12},$ where the latter is the square root of $a^{\dag}$ in $E_2(e)$.\smallskip

Since $ \|a - b\| < \gamma^{q}(a)^{\frac{1}{2}}$, then $$\left\| e - Q(c^{\dag}) ( P_2 (e) (b))\right\| = \left\| Q(c^{\dag}) ( a-b )\right\|\leq \|Q(c^{\dag})\| \|a-b\| < \|c^{\dag}\|^2 (\gamma^{q}(a))^{\frac{1}{2}}$$

$$= \|a^{\dag}\|\ \gamma^{q}(a)^{\frac{1}{2}} =\hbox{(see \cite[Theorem 3.4 and its proof]{BurKaMoPeRa})}= 1.$$ Since $e$ is the unit of $E_2 (e)$, we deduce that $Q(c^{\dag}) ( P_2 (e) (b))$ is invertible in $E_2(e).$ It is well known from the theory of invertible elements in JB$^*$-algebras that $Q(c^{\dag})|_{E_2(e)}: E_2 (e)\to E_2 (e)$ is invertible as a mapping from $E_2(e)$ into itself with inverse $Q(c)|_{E_2(e)}: E_2 (e)\to E_2 (e)$ (cf. \cite[\S 4.1.1]{CabRod2014}). Since $Q(c^{\dag}) ( P_2 (e) (b))$ is invertible, we deduce that $P_2 (e) (b) = Q(c) Q(c^{\dag}) ( P_2 (e) (b))$ is invertible in $E_2 (e)$ (cf. \cite[Theorem 4.1.3]{CabRod2014}). Finally, Lemma 2.2 in \cite{JamPerSiTah2015} implies that $b\in E_q^{-1}$, as we desired.
\end{proof}

\begin{prop}\label{p suficient conditions for continuity of seminorm} Let $a$ be a BP quasi-invertible element in a JB$^*$-triple $E$. Suppose $b$ is an element in $E$ such that  $\|a - b\| < \gamma^{q}(a)^{\frac{1}{2}}$. Then $b\in E_q^{-1}$ and  $$\|\gamma^{q}(a) - \gamma^{q}(b)\| < \gamma^{q}(a)^{\frac{1}{2}} \|a - b\|.$$
\end{prop}

\begin{proof} Lemma \ref{l 31} implies that $b$ is BP quasi-invertible. It is known that $$\gamma^{q}(x) = m_q (x)^{2} \leq \|x\|^2,$$ and $| m_q (x) -m_q (y)| \leq \|x-y\|$ for every $x,y\in E$ (cf. \cite[\S 3]{JamPerSiTah2015}).  Therefore, $$|\gamma^{q}(a) - \gamma^{q}(b)| = |m_{q}(a)^{2} - m_{q}(b)^{2}| = | m_{q}(a) - m_{q}(b)| \ |m_{q}(a) + m_{q}(b)|$$ $$ < \|a - b\| (\|a\| +\|b\|)< \gamma^{q}(a)^{\frac{1}{2}}
\|a - b\|.$$
\end{proof}

It is proved in \cite[Theorem 3.13]{BurKaMoPeRa} that the quadratic conorm, $\gamma^{q}(.)$, in a JB$^*$-triple $E$ is upper semi-continuous on $E\backslash \{0\}$. In the setting of extremally rich JB$^*$-triples we can characterize now the precise points at which $\gamma^{q}(.)$ is continuous.

\begin{thm}\label{t characterization of continuity points of conorm} Let $E$ be an extremally rich JB$^*$-triple. Then the quadratic conorm $\gamma^{q}(.)$ is continuous at a point $a\in E$ if, and only if, either $a$ is not von Neumann regular {\rm(}i.e. $\gamma^{q}(a)=0${\rm)} or $a$ is BP quasi-invertible.
\end{thm}

\begin{proof} The upper semi-continuity of $\gamma^{q}(.)$ implies that it is continuous at every point $a\in E$ which is not von Neumann regular. If $a \in E^{-1}_{q}$, the continuity of $\gamma ^{q}(.)$ at $a$ follows from Proposition \ref{p suficient conditions for continuity of seminorm}.\smallskip

Suppose that $\gamma ^{q}(.)$ is continuous at $a$, and $a$ is von Neumann regular (i.e. $\gamma^{q}(a)>0$). In this case, $Q(a)(E)$ is norm closed, or equivalently, $\gamma^{q}(a) = \gamma(Q(a)) > 0$ (see \cite[Corollary 2.4 and proof of Theorem 3.4]{BurKaMoPeRa}). The mapping $x\mapsto \gamma^{q}(x)^{\frac12}$ is continuous at $a$. So, there exists $\delta > 0 $ such that $$\|a - b\| < \delta \Rightarrow |\gamma^{q}(a)^{\frac{1}{2}} - \gamma^{q}(b)^{\frac{1}{2}}| < \frac{\gamma^{q}(a)^{\frac{1}{2}}}{2},$$ that is, $ \gamma^{q}(b)^{\frac{1}{2}} > \frac{\gamma^{q}(a)^{\frac{1}{2}}}{2},$ whenever $\|a - b\| < \delta$. Extremally richness of $E$, implies that $\overline{E^{-1}_{q}} = E$. Thus, there is $c \in E^{-1}_{q}$ with $\|a - c\| < \min\{\delta, \frac{\gamma^{q}(a)^{\frac{1}{2}}}{2}\}$.  In particular $\|a - c\| < \delta$, that is, $\gamma^{q}(c)^{\frac{1}{2}} >  \frac{\gamma^{q}(a)^{\frac{1}{2}}}{2}
 > \|a - c\|$. Lemma \ref{l 31} above proves that $ a \in E^{-1}_{q}$.
 \end{proof}

\begin{remark}\label{r continuity points of the qnorm} In \cite[Remark 3.18]{BurKaMoPeRa} it is shown that the quadratic conorm $\gamma^q (.)$ of a JB$^*$-triple $E$ is continuous at every element $a\in E$ for which $Q(a)$ is left or right invertible in $B(E)$. It is also asked, in the just quoted remark, whether these points are the only non-trivial continuity points of $\gamma^q (.)$. Theorem \ref{t characterization of continuity points of conorm} characterizes the continuity points of the quadratic conorm in the class of extremally rich JB$^*$-triples (a class including all JBW$^*$-triples). Proposition \ref{p suficient conditions for continuity of seminorm} shows the existence of points $x$ satisfying that the quadratic conorm is continuous at $x$, but $Q(x)$ is not left nor right invertible. For example, when $E$ is an extremally rich JB$^*$-triple and $e$ is a complete tripotent with $E_1(e)\neq \{0\}$, then the quadratic conorm is continuous at $e$, but $Q(e)$ is not left nor right invertible.
\end{remark}

The arguments in the second part of the proof of Theorem \ref{t characterization of continuity points of conorm} are also valid to prove the following:

\begin{prop}\label{p necessary conds continuity qconorm} Let $(a_{n})$ be a sequence of BP quasi-invertible elements in a JB$^*$-triple $E$. Suppose that $(a_{n})$ converges in norm to some element $a$ in $E$, and $\gamma^q (a_n) \to \gamma^q (a)>0$. Then $a$ is BP quasi-invertible.$\hfill\Box$
\end{prop}

Our next result is a consequence of \cite[Theorem 3.16 and Corollary 3.17]{BurKaMoPeRa} and the previous Proposition \ref{p necessary conds continuity qconorm}.

\begin{corollary}
Let $(a_{n})$ be a sequence of BP quasi-invertible elements in a JB$^*$-triple $E$. Suppose that $(a_{n})$ converges in norm to some element $a$ in $E$. Then the following assertions are equivalent:\begin{enumerate}[$(a)$]
\item $(a_n^{\dag})$ is a bounded sequence in $E;$
\item $\gamma^{q}(a_{n}) \to \gamma^{q}(a) > 0$;
\end{enumerate}
Furthermore, if any of the above statements holds, then $a$ is BP quasi-invertible and $\|a_n^{\dag}-a^{\dag}\|\to 0.$
\end{corollary}

\begin{proof}$(a)\Rightarrow (b)$ Suppose $(a_n^{\dag})$ is a bounded sequence in $E.$ Corollary 3.17 in \cite{BurKaMoPeRa} implies that $a$ is von Neumann regular (i.e. $\gamma^{q}(a)>0$). It follows from \cite[Theorem 3.16]{BurKaMoPeRa}$(d)\Rightarrow (c)$ that $\gamma^{q}(a_{n}) \to \gamma^{q}(a)>0$.\smallskip

$(b)\Rightarrow (a)$ Suppose $\gamma^{q}(a_{n}) \to \gamma^{q}(a)>0$. In particular, $a$ is von Neumann regular. The desired statement follows from \cite[Theorem 3.16]{BurKaMoPeRa}$(c)\Rightarrow (d)$.\smallskip

The final statement is a consequence of Proposition \ref{p necessary conds continuity qconorm} and \cite[Theorem 3.16]{BurKaMoPeRa}.
\end{proof}

The result in Theorem \ref{t characterization of continuity points of conorm} is new, even in the case of C$^*$-algebras. R. Harte and M. Mbekhta introduce the notions of left and right conorms for C$^*$-algebras in \cite{HarMb2}.  According to their notation, the
left conorm, $\gamma(a),$ of an element $a$ in a C$^*$-algebra $A$ is given by
$$ \gamma(a)=\gamma^{left} (a)=\gamma(L_a)=\inf \left\{ \frac{\Vert a
x\Vert}{{\rm d}(x,\ker(L_a))} : x\not\in \ker(L_a)\right\},$$ where $L_a$ is the left multiplication mapping by $a$, that is, $L_a (x) = a x$ ($x\in A$). The right conorm is similarly defined. Theorem 4 in \cite{HarMb2} shows that $$\gamma (a)^2 = \gamma(aa^*) = \gamma (a^*a) =\gamma^{right} (a)^2 = \inf\{t : t\in \sigma(a a^*)\backslash \{0\}\}, $$ where $\sigma(a a^*)$ denotes the spectrum of $aa^*$.\smallskip

Harte and Mbekhta established that the conorm $\gamma(.)$ of a C$^*$-algebra is upper semi-continuous \cite[Theorem 7]{HarMb2}, they also show in \cite[Theorem 9]{HarMb2} that the conorm is always continuous at semi-invertible elements, and, by the upper semi-continuity of $\gamma(.)$, the conorm is continuous at elements with no generalized inverses (i.e. at elements $a$ with $\gamma (a)=0$). When $A= B(H)$, the C$^*$-algebra of all bounded linear operators on a complex Hilbert space $H$, then these results cover all continuity points. The general case is left as open problem. For a general C$^*$-algebra $A$, Corollary 4.1 in \cite{BurKaMoPeRa} proves that $\gamma^{q} (a) = \gamma (a)^2,$ for all $a\in A.$ Theorem \ref{t characterization of continuity points of conorm} particularizes in the following result, which provides an additional information to the problem left open by Harte and Mbekhta.

\begin{corollary} Let $A$ be an extremally rich C$^*$-algebra. Then the conorm of $A$ is continuous at a point $a\in A$ if, and only if, either $a$ is not von Neumann regular {\rm(}i.e. $\gamma(a)=0${\rm)} or $a$ is quasi-invertible.$\hfill\Box$
\end{corollary}

\end{document}